\documentclass[11pt,a4paper,reqno]{amsart}
\usepackage{epsfig,amssymb,latexsym,amscd,amsmath,amsthm,stmaryrd}
\usepackage{verbatim}
\usepackage[sans]{dsfont}
\usepackage{mathabx}
\usepackage{yfonts}
\usepackage{stmaryrd}
\usepackage{euscript}
\usepackage{mathrsfs}
\usepackage{txfonts}

\theoremstyle{plain}

\newtheorem{teo}{Theorem}[section]
\newtheorem{theo}[teo]{Theorem}
\newtheorem{coro}[teo]{Corollary}

\theoremstyle{remark}

\newtheorem{rema}[teo]{Remark}
\newtheorem{rems}[teo]{Remarks}

\theoremstyle{definition}

\newtheorem{defi}[teo]{Definition}
\newtheorem{obse}[teo]{Observation}

\newtheorem{example}[teo]{Example}

\newcommand{\compo}{\,{\scriptstyle \circ }\,}

\newcommand{\id}{\ensuremath{\mathrm{id}}}

\newcommand{\ant}{\mathcal{S}}

\renewcommand{\k}{{\mathbb K}}
\newcommand{\K}{{\mathbb K}}

\newcommand{\C}{{\mathbb C}}

\newcommand{\tl}{\triangleleft}
\newcommand{\tr}{\triangleright}

\newcommand{\rhu}{{\rightharpoonup}}
\newcommand{\lhu}{{\leftharpoonup}}

\newcommand{\cop}{\ensuremath{\mathrm{cop}}}

\begin{document}

\title[Almost involutive Hopf algebras]
{Almost involutive Hopf algebras}

\author{Andr\'es Abella}
\address{Facultad de Ciencias\\Universidad de la Rep\'ublica\\
Igu\'a 4225\\11400 Montevideo\\Uruguay\\}
\email{andres@cmat.edu.uy}
\author{Walter Ferrer Santos}
\thanks{The second author would like to thank Csic-UDELAR,  Conycit-MEC, Uruguay and Anii, Uruguay.}
\email{wrferrer@cmat.edu.uy}
\begin{abstract} We define the concept of \emph{companion automorphism} of a Hopf algebra $H$ as an automorphism 
$\sigma:H \rightarrow H$: $\sigma^2=\ant^2$ --where $\ant$ denotes the antipode--.  
A Hopf algebra is said to be \emph{almost involutive} (AI) if it admits a companion automorphism that can be viewed as a special additional symmetry. We present examples and study some of the basic properties and constructions 
of AI-Hopf algebras centering the attention 
in the finite dimensional case. In particular we show that within the family of Hopf algebras of dimension smaller or equal than 15, 
only in dimension eight and twelve, there are non almost involutive Hopf algebras.   
\end{abstract}
\maketitle

\section{Introduction.}

We start by summarizing the contents of this paper. 

\bigskip
\noindent
In Section \ref{section:defejemplos} we introduce the definition of an \emph{almost involutive Hopf algebra}, show that 
Sweedler's Hopf algebra is an example and observe that compact quantum groups are also examples --in the infinite 
dimensional case--. Then we show that some of the standardt constructions, such as 
that of a matched pair or in particular of the Drinfel'd double, when applied to almost involutive Hopf algebras yield as result an almost involutive Hopf algebra. 
Moreover, as in the case that the order of the antipode squared is odd the AI property is authomatic, we concentrate in the interesting case that is when the order of the antipode squared is even.  

\bigskip
\noindent
In Section \ref{section:prelim} we recall some aspects of the theory of integrals in finite dimensional Hopf algebras,
the modular function, the modular element, etc. These tools will be used in the rest of the paper. 

\bigskip
\noindent
In Section \ref{section:examples} we present some examples and non--examples of almost involutive Hopf algebras, and show that 
except at dimension 8 and 12 and for only a few types, all Hopf algebras of dimension smaller or equal than 15 are almost 
involutive. 

\bigskip
\noindent
In the Appendix, we present --a probably well known-- result on square roots of linear automorphisms of Hopf algebras that yield necessary 
and sufficient conditions for a finite order automorphism to have a square root that is also a Hopf automorphism. This general result can be
used in some examples, for example we use it to show that the Taft algebras are almost involutive. 

\bigskip
\noindent
We finish this introduction with some commentaries about notations.

\bigskip
\noindent
We  work over an algebraically closed field $\k$ of characteristic zero and adopt the usual Sweedler's notation and the other conventions
in the theory of Hopf algebras as they appear for example in \cite{kn:mont}, e.g. $\ant$ denotes the antipode, and $\lambda$ is an integral, .
Also a non zero element $g$ in $H$ is called a \emph{group-like} element if $\Delta(g)=g\otimes g$ and we 
designate as $G(H)$ and $G(H^\vee)$ the set of group-like elements in the original Hopf algebra and its dual. Moreover, if $x\in H$ 
is such that $\Delta(x)=x\otimes g+ h\otimes x$, where $g,h\in G(H)$, then $x$ is called a {$(g,h)$-primitive} element and we write as $P_{g,h}$ the space of $(g,h)$-primitive elements oh $H$.
A Hopf algebra is \textit{pointed} if its only simple subcoalgebras are the dimensional.

\bigskip
\noindent
In general we concentrate our considerations in the case that finite dimensional Hopf algebras.

\section{Main definition and general examples.}\label{section:defejemplos}
\subsection{General considerations}
Recall that a Hopf algebra  $H$ is called \textit{involutive} if $\ant^2=\id$. Examples of involutive Hopf algebras are commutative, cocommutative or semisimple Hopf algebras over $\C$.

\begin{defi}
We say that a Hopf algebra $H$ is \textit{almost involutive} or that is an \textit{AI-Hopf algebra} if there exists a Hopf algebra automorphism 
$\sigma:H\to H$ such that $\ant^2=\sigma^2$. An automorphism $\sigma$ as above, is called a {\em companion automorphism} of $\ant$ or simply a {\em companion automorphism} of $H$. 
\end{defi}

It follows from Radford's formula --see \cite{kn:bbt}, \cite{kn:radfordbasic} and \cite{kn:schcordoba}-- that $\ant^2$ is a diagonalizable 
Hopf algebra automorphism of finite order, more precisely: $\operatorname{ord}(\ant^2)| 2\operatorname{ord}(a)\operatorname{ord}(\alpha)$ --where $a$ is the modular element and $\alpha$ the modular form, see Section \ref{section:prelim}. 
It can also be proved that if $H_n$ is the coradical filtration of $H$, then 
$\operatorname{ord}{\ant^2}|_{H_1}=\operatorname{ord}{\ant^2}$ --see \cite[Lemma 4]{kn:radfordsch}--. 

The observation that follows shows that if the antipode squared is of finite order $m$ 
--in particular if $H$ is finite dimensional--,  then the only case in which it is interesting to consider the AI condition is when $m$ is even. 
 
\begin{obse} \label{obse:oddtrivial}
Let $H$ be a Hopf algebra with order of $\ant^2$ odd. Then $H$ is almost involutive.
Indeed, if $(\ant^2)^{2k-1}=\operatorname{id}$, then $\ant^2=(\ant^{2k})^2$. 
\end{obse}

\begin{example}\label{exam:first}
\emph{Sweedler's Hopf algebra.}\label{example:sweedler}
As an illustration of an almost involutive Hopf algebra with the antipode squared of even order, we take Sweedler's Hopf algebra $H_4$. As an 
algebra $H_4$ is given by generators and relations as $H_4=\langle g,x:\ g^2=1,\ x^2=0,\ xg+gx=0\rangle$. 
The comultiplication is given by $\Delta(g)=g\otimes g$, $\Delta(x)=x\otimes g+1\otimes x$, and the antipode is  defined by 
$\ant(g)=g$, $\ant(x)=-xg$. Clearly, $\ant^2(g)=g$ and $\ant^2(x)=-x$ and $|\ant^2|=2$. 
\noindent
A direct verification shows that the map $\sigma$ defined by $\sigma:g \mapsto g$, $\sigma: x \mapsto ix$ and extended multiplicatively, 
is a companion automorphism where $i$ is a square root of $-1$. 
The map: $g \mapsto g,\, x \mapsto -ix$ is also a companion automorphism. 
\end{example}

The following result shows that almost involutive non involutive Hopf algebras abound. 

\begin{teo}
Let $H$ be a finite dimensional pointed Hopf algebra. 
If the order $|G(H)|$ is odd, then $H$ is almost involutive. 
In particular this happens if the dimension of $H$ is odd. 
\end{teo}
\begin{proof}
We know $\operatorname{ord}{\ant^2}=\operatorname{ord}{\ant^2}|_{H_1}$, being $H_n$ the coradical filtration of $H$.
As $H$ is pointed, then $H_1=\K G(H)\oplus\oplus_{g,h\in G(H)}P'_{g,h}(H)$,  where $P'_{g,h}(H)$ is an arbitrary linear complement of $\k(g-h) \subseteq P_{g,h}(H)$ --see \cite[Thm. 0.1]{kn:stefan}--.
It is easy to see that $\ant^2|_{G(H)}=\id$ and if $x\in P_{g,h}$, then $\ant^2(x)=gh^{-1}xg^{-1}h$.
This last formula implies $\ant^{2n}(x)=\left(gh^{-1}\right)^nx\left(g^{-1}h\right)^n$, for all $n$.
Hence $\operatorname{ord}{\ant^2}$ divides $|G(H)|$ and we can apply the observation \ref{obse:oddtrivial}.
The last assertion follows because $|G(H)|$ divides the dimension of $H$.
\end{proof}

If the order of $|G(H)|$ is even, then the above theorem is no longer true -- see examples \ref{ej:8} and \ref{ej:12} below--. 

\bigbreak

Concerning the infinite dimensional situation, we mention that in \cite{kn:andres-walter-mariana}, the authors --together with M. Haim-- 
proved the following general theorem. 

\begin{teo}\emph{[\cite{kn:andres-walter-mariana}, Theorem 8.7.]} 
Let $H$ be a compact quantum group with antipode $\ant$. 
Then, $H$ admits a unique companion automorphism --that in this case is denoted as $\ant_+$ instead of $\sigma$--, that is a positive operator with respect to the natural inner product of $H$, 
and there is a linear functional $\beta: H \rightarrow \mathbb C$ such that $\ant_+(x)=\sum \beta(x_1)x_2\beta^{-1}(x_3)$. 
\end{teo}

In the mentioned paper the following explicit example is constructed.

\begin{example} See Woronowicz's \cite{kn:worosu2}. Let $\mu \in \mathbb C$ be a non zero real number such that $|\mu|<1$, and call 
$\operatorname{SU}_\mu(2,\mathbb C)$ the algebra generated by $\{\alpha,\alpha^*, \gamma, \gamma^*\}$, subject to the following relations:
$\alpha^*\alpha+\mu^2 \gamma^*\gamma=1$, $\alpha \alpha^*+\mu^4 \gamma \gamma^*=1$, $\gamma^*\gamma=\gamma\gamma^*$, 
$\mu\gamma\alpha=\alpha\gamma$, $\mu\gamma^*\alpha=\alpha\gamma^*$, $\mu^{-1}\gamma\alpha^*=\alpha^*\gamma$, 
$\mu^{-1}\gamma^*\alpha^*=\alpha^*\gamma^*$. The star structure in $\operatorname{SU}_\mu(2,\mathbb C)$, is defined as being antimultiplicative, 
involutive, conjugate linear and defined on the generators as shown. This algebra admits a natural compatible comultiplication and becomes a
\emph{compact quantum group}
--see \cite{kn:worosu2} for details-- with antipode: $\ant(\alpha)=\alpha^*, \ant(\alpha^*)= \alpha, \ant(\gamma)
=-\mu \gamma, \ant(\gamma^*)=-\mu^{-1} \gamma^*$.  One shows that the companion automorphism $\ant_+$ is given as: 
$\ant_+(\alpha)=\alpha$, $\ant_+(\alpha^*)=\alpha^*$, $\ant_+(\gamma)=|\mu|\gamma$, $\ant_+(\gamma^*)=|\mu|^{-1}\gamma^*$. 
\end{example}

\begin{obse}
\begin{enumerate}
\item Recall that if $\tau:H \rightarrow H$ is a Hopf algebra homomorphism,  then it commutes with $\ant$. In particular 
in the situation above the antipode $\ant$ and the companion automorphism $\sigma$ commute.
\item In general one cannot guarantee the uniqueness of $\sigma$. See the comment at the end of Example \ref{exam:first}. 
\item More examples and non examples will be provided in later sections, here we mention the following: involutive Hopf algebras and the 
quantum enveloping algebra of $\mathfrak{sl} (2)$. 
\end{enumerate}
\end{obse}

\subsection{Constructions of almost involutive Hopf algebras.} \label{subsection:const}\quad

\smallskip
\noindent
The following constructions always produce AI-Hopf algebras.

\begin{enumerate}
\item \emph{Duals and tensor products.} If $H$ is a finite dimensional AI-Hopf algebra, so is its dual. If $H$ and $K$ are AI-Hopf algebras so is $H \otimes K$.   
\item \emph{Matched pairs.} Assume that we have a matched pair of Hopf algebras $(A,H,\tl,\tr)$. If $A$ and $H$ are almost involutive with companion morphisms $\sigma_A$ and $\sigma_H$, so is the bicrossed product $A \bowtie H$ provided that the given companion automorphisms for $A$ and $H$ satisfy the following compatibility conditions:
 \begin{enumerate}
 \item $\sigma_A(x \tr a)= \sigma_H(x) \tr \sigma_A(a)$ 
 \item $\sigma_H(x \tl a)= \sigma_H(x) \tl \sigma_A(a)$.
 \end{enumerate}

This assertion follows easily form the fact that the antipode for the bicrossed product is simply $\ant(ax)=\ant_H(x)\ant_A(a)$ for $a \in A, x \in H$ --see for example \cite{kn:andres-walter-mariana-2} or \cite{kn:Majid} for a detailed description of the structure of $A \bowtie H$--. Hence if we write $\ant_A^2=\sigma_A^2$ and  $\ant_H^2=\sigma_H^2$, the Hopf algebra morphism $\sigma(ax)=\sigma_A(a) \sigma_H(x)$ does the job.
\item \emph{Drinfel'd double.} In particular, if $H$ is almost involutive, so is its Drinfel'd double $D(H)$. This follows from the fact that the Drinfel'd double can be viewed as a bicrossed product (see \cite{kn:andres-walter-mariana-2} or \cite{kn:Majid}) $D(H)=H^{\vee\cop} \bowtie H$ where the structure of matched pair is given by  
$H \stackrel{\triangleleft}{\leftarrow} H\otimes H^\vee\stackrel{\triangleright}{\to}H^\vee$ defined as:
\begin{equation*} 
(x\tr\alpha)(y)=\sum \alpha\left(\ant^{-1}(x_2) y x_1\right),\quad
x\tl \alpha = \sum \alpha\left(\ant^{-1}(x_3)x_1 \right) x_2
,\quad \forall \alpha\in H^\vee,\ x,y\in H.
\end{equation*}  

In this situation one easily verifies that if $\sigma$ is a companion automorphism for $H$, its natural extension to $D(H)$ is a companion automorphism for the double.
\item \label{item:trivial} \emph{Trivial extensions.} 
 \begin{enumerate}
 \item Consider a Hopf algebra $H$ that can be written $H=K \oplus M$
where $K$ and $M$ satisfy the following conditions:
  \begin{enumerate}
  \item $K$ is a sub Hopf algebra of $H$.
  \item $M$ is a $K$--bimodule and a $K$--bicomodule, in other words the following holds:
   \begin{enumerate}
   \item $KM + MK \subseteq M$;
   \item $\Delta(M) \subseteq K \otimes M + M\otimes K$;
   \end{enumerate}
  \item $M$ is invariant by $\ant$.
  \item The extension of $K$ by $M$ is trivial, in other words, $M^2=0$.
  \end{enumerate}
If $m \in M$ we have that $\Delta(m)=\sum h_i \otimes m_i + n_i \otimes k_i$ with $h_i,k_i \in K$ and $m_i,n_i \in M$. Then $m=\sum \varepsilon(m_i)h_i + \varepsilon(k_i)n_i$. If the $h_i$ are linearly independent, we deduce that $\varepsilon(m_i)=0$ and similarly for $n_i$. Hence, if we write write the expression $\Delta(m)=\sum h_i \otimes m_i + n_i \otimes k_i$ with $h_i$ and $k_i$ linearly independent, and apply $\varepsilon \otimes \varepsilon$ to the expression obtain that $\varepsilon(m)=0$. Hence $\varepsilon(M)=0$. 
Assume that $K$ is almost involutive, and call $\sigma_K$ the corresponding companion automorphism.  Suppose also that we can find a linear map $\sigma_M$, with the property that $\sigma_M^2=\ant^2|_M$ and such that:
 \begin{enumerate}
 \item $\sigma_M(hm)=\sigma_K(h)\sigma_M(m)\,,\,\sigma_M(mh)=\sigma_M(m)\sigma_K(h)$.
 \item If $\Delta(m)=\sum h_i \otimes m_i + n_i \otimes k_i \in K \otimes M + M \otimes K$, then $\Delta(\sigma_M(m))=\sum \sigma_K(h_i) \otimes \sigma_M(m_i) + \sigma_M(n_i) \otimes \sigma_K(k_i)$. 
\end{enumerate}
In that situation the map $\sigma(h + m)=\sigma_K(h)+\sigma_M(m)$ is a companion automorphism in $H$. First, it is clear that $\sigma^2=\ant^2$. Morover $\sigma((h+m)(k+n))=\sigma(hk+hn+mk)=\sigma_K(hk)+\sigma_M(hn+mk)=
\sigma_K(h)\sigma_K(k)+ \sigma_K(h)\sigma_M(n)+\sigma_M(m)\sigma_K(k)= (\sigma_K(h)+\sigma_M(m))(\sigma_K(k)+\sigma_M(n))=\sigma(h+m)\sigma(k+n)$. 

\noindent
Also, a direct verification shows that the map $\sigma$ is an automorphism of coalgebras.  

 \item 
Consider the particular case of a trivial extension such that $K=\left\{x\in H:\ \ant^2(x)=x\right\}$. 
Using the techniques developped in the Appendix --see Example \ref{example:general}-- one can prove that in this situation, $H$ is an almost involutive Hopf algebra. 

Sweedler's Hopf algebra is an example. Clearly, $H_4$ can be written in the above manner with: $K=\k+\k g$ and $M=\k x + \k gx$. 
In this situation $\sigma_K=\operatorname{id}$ and $\sigma_M=i\operatorname{id}$ satisfies all the required properties. Other examples of the above situation will appear later --see Section \ref{section:examples}--. 
 \end{enumerate}
\end{enumerate}

\section{Basic results on finite dimensional Hopf algebras.} \label{section:prelim} 

Here we recall basic results and constructions concerning Hopf algebras that will be used later in the study of some of the examples 
--see \cite{kn:Dascalescu} and \cite{kn:radfordbasic} for proofs and details. We concentrate on the basic properties of integrals, 
modular function and modular element. 

There exists $\lambda \in H^\vee$ and $\ell \in H$ that satisfy $\lambda(\ell)=1$, and  invertible elements $a\in G(H)$ and 
$\alpha\in \text{Alg}(H,\K)$ such that for all $x \in H$:
\begin{equation}\label{eq:intl}
\sum x_1\lambda(x_2)=\lambda(x)1\,\quad 
x\ell=\varepsilon(x)\ell,\quad
\sum \lambda(x_1)x_2=\lambda(x)a,\quad
\ell x=\alpha(x)\ell.
\end{equation}

The elements $\lambda$ and $\ell$ that are \emph{left integrals} and the elements $a$ and $\alpha$ are called 
\textit{modular element} and \textit{modular function} respectively. 

\begin{enumerate}
\item 
The modular function and the modular element are of finite order in $H^\vee$ and $H$ respectively and then $\alpha(a)$ is a root of one of 
order less or equal than the dimension of $H$. 
In fact $\operatorname{ord}(\alpha(a))|\gcd\{\operatorname{ord}(\alpha),\operatorname{ord}(a)\}$. 
\item 
The elements $\lambda \compo \ant= \rho$ and $\ant^{-1}(\ell)=r$ are \emph{right integrals} such that $\rho(r)=1$. We have that for all $x \in H$:
\begin{equation}\label{eq:rint}
\sum \rho(x_1)x_2=\rho(x)1,\quad
rx=\varepsilon(x)r,\quad 
\sum x_1\rho(x_2)=\rho(x)a^{-1},\quad 
xr=\alpha^{-1}(x)r.
\end{equation}   
\item 
The elements $\lambda$ and $\ell$ ($\rho$ and $r$) are uniquely determined up to a non zero scalar.
\item
The Hopf algebra $H$ (or $H^\vee$) is unimodular, i.e. a left integral is also a right integral, if and only if $\alpha=\varepsilon$ 
(or $a = 1$) respectively. 
\end{enumerate}

\begin{rems} The following holds --see \cite{kn:schcordoba}--:
\begin{enumerate}
\item 
\begin{equation}\label{eq:antipode}  
\ant(x)=\sum \ell_1 \lambda(x\ell_2), \quad \forall x\in H .
\end{equation}

\item 
From the formula \eqref{eq:antipode}, we can easily deduce that: 
\begin{equation}\label{eq:lambdaS}
\ant(\ell)= \sum \ell_1 \alpha(\ell_2)
\quad\text{and}\quad
\lambda\big(\ant(x)\big) = \lambda(xa),  \quad \forall x\in H . 
\end{equation}
Indeed, applying $\lambda$ we have that: $\lambda(\ant(x))=\sum \lambda(\ell_1) \lambda(x\ell_2)=\sum \lambda(x\lambda(\ell_1)\ell_2) = \lambda(xa)$. The other formula is the dual.
\item 
In the situation above, we have that: 
\begin{eqnarray} \label{eq:bilinear} 
\ant^2(\ell)= \alpha(a) \ell,\quad  
\ant^{2}(r)= \alpha(a) r, \quad
\lambda \compo \ant^2=\alpha(a) \lambda,\quad 
\rho \compo \ant^2=\alpha(a) \rho . 
\end{eqnarray}
\noindent
Indeed, by iteration of the formula \eqref{eq:lambdaS}  we obtain that $\lambda\/\compo\/\ant^2 = a \rhu \lambda \lhu a^{-1}$. 
Being $\lambda \compo \ant^2$ another left integral, we conclude that it has to be a scalar multiple of $\lambda$. 
As a $(a \rhu \lambda \lhu a^{-1})(\ell)=\lambda(a^{-1}\ell a)=\varepsilon(a^{-1}) \lambda(\ell) \alpha(a)=\alpha(a)$ we conclude that $\lambda \compo \ant^2=\alpha(a) \lambda$.  
The other formul\ae\/ are proved similarly. The formul\ae\/ for right integrals are obtained by composition with $\ant^{\pm 1}$. 
Notice that both $\ell$ and $r$ are eigenvectors of $\ant^2$ with  eigenvalue $\alpha(a)$. 
\item 
In particular $\lambda(\ant(\ell))=\alpha(a)$ and $\lambda(r)=1$. 
Indeed, if we put $x=\ell$ in the equation \eqref{eq:lambdaS} we obtain that $\lambda(\ant(\ell))= \lambda(\ell a)= \lambda\big(\alpha(a)\ell\big)=\alpha(a)$. 
Moreover, applying  the equality $\lambda \circ \ant^2=\alpha(a) \lambda$ to the element $r=\ant^{-1}(\ell)$ we deduce that: $\lambda(r)=1$. 
\end{enumerate}
\end{rems}

Next we look at the behaviour of the elements above when transformed with a companion automorphism.

\begin{obse} Let $\sigma$ be a companion automorphism of $H$.
\begin{enumerate}
\item The following holds because $\sigma$ is a Hopf algebra map --the proof is omitted as it is standard--:
\begin{gather}
\label{eqn:first}
\sigma(\ell) =\lambda(\sigma(\ell))\ell ,\quad
\sigma(r)=\rho(\sigma(r))r,\quad
\lambda \circ \sigma = \lambda(\sigma(\ell))\lambda ,\quad 
\rho \circ \sigma = \rho(\sigma(r))\rho , 
\\
\label{eq:alfasigma}
\sigma(a)=a,\quad
\sigma\left(a^{-1}\right)=a^{-1},\quad
\alpha \circ \sigma = \alpha ,\quad  
\alpha^{-1} \circ \sigma = \alpha^{-1}. 
\end{gather}
\item 
From the definition of $\rho$ and $r$ and using that $\sigma$ is a Hopf algebra map we deduce $\rho(\sigma(r))=\lambda(\sigma(\ell))$.
Then using \eqref{eq:bilinear} and \eqref{eqn:first}  we deduce $\lambda(\sigma(\ell))^2=\alpha(a)$.
Hence we have that:
\begin{align}\label{eq:erresigma}
\sigma(\ell)=r_\sigma \ell, \quad   
\sigma(r)=r_\sigma r  ,\quad 
\lambda \circ \sigma= r_\sigma \lambda ,\quad 
\rho \circ \sigma= r_\sigma \rho
\quad \text{with} \quad r_\sigma^2=\alpha(a).
\end{align} 
\item 
We consider the Sweedler's algebra. 
The set $\{1,g,x,gx\}$ is a basis of $H_4$ and we will denote as $\{1^*,g^*,x^*,(gx)^*\}$ its dual basis. 
We get
\[
\ell= (1+g)x;\ r=x(1+g);\  \lambda= x^*; \ \rho=-(gx)^*;\   a=g\,\,;\,\, \alpha:H_4 \rightarrow \K,\ \text{is given by } \alpha(g)=-1,\ \alpha(x)=0.
\]  
Recall that the map $\sigma:H_4 \rightarrow H_4$, defined as $\sigma(g)=g$ and $\sigma(x)=ix$ and extended multiplicatively, is a companion automorphism for $H_4$. 
As $\sigma(\ell)=i\ell$ 
, in this case $r_\sigma=i$. 

\item 
Call $E_{\sigma,\nu}$ the eigenspace of $\sigma$ corresponding to the eigenvalue $\nu$. Then:

\begin{gather}
\ell,r \in E_{\sigma,r_\sigma},\quad 1,a,a^{-1} \in E_{\sigma,1}, \nonumber\\
\bigoplus_{\nu \neq r_\sigma}E_{\sigma,\nu} \subset \operatorname{Ker}(\lambda), \nonumber \\
\label{eqn:fifth}
\bigoplus_{\nu \neq 1}E_{\sigma,\nu} \subset \operatorname{Ker}(\alpha)\cap \operatorname{Ker}\left(\alpha^{-1}\right)\cap \operatorname{Ker}(\varepsilon), \\
E_{\sigma,\nu} \ell = \ell E_{\sigma,\nu} =0,\quad \nu\ne 1. \nonumber
\end{gather} 
First observe that the equations \eqref{eq:alfasigma} and \eqref{eq:erresigma} mean that $\ell,r \in E_{\sigma,\lambda(\sigma(\ell))}$ and $a,a^{-1} \in E_{\sigma,1}$, and is clear that $1 \in E_{\sigma,1}$. 
Moreover, if $x \in E_{\sigma,\nu}$, applying  respectively $\lambda,\alpha,\alpha^{-1},\varepsilon$ to the equality $\sigma(x)=\nu x$, we deduce that 
$r_\sigma\lambda(x)= \nu \lambda(x)$, $\alpha(x)=\nu \alpha(x)$, $\alpha^{-1}(x)=\nu \alpha^{-1}(x)$ and $\varepsilon(x)=\nu\varepsilon(x)$. 
From the first of these equalities we deduce that if $\nu\ne r_\sigma$, then $\lambda(x)=0$, and similarly for the others. 
For the last assertion: if $x \in E_{\sigma,\nu}$, from the equality $x\ell=\varepsilon(x)\ell$ and \eqref{eqn:fifth} we deduce that $x\ell=0$. Similarly we deduce that $\ell x=0$. 
\end{enumerate}
\end{obse}

\section{Description of the almost involutive Hopf algebras up to dimension 15.}\label{section:examples}

In the following examples we will often have to deal with a Hopf algebra $H$ with a given group-like element $g$ and a $(g,1)$-primitive element $x$, then $\ant(g)=g^{-1}$, $\ant(x)=-xg^{-1}$,  $\ant^2(g)=g$ and $\ant^2(x)=gxg^{-1}$.

\begin{example} [Example 1, in \cite{kn:radfordbasic}] \label{example:rad}
Let be $\omega\in\k$ a primitive root of order $n$ of 1 and 
\[
H=\left\langle g,x,y:\ g^n=1,\ x^n=0,\ y^n=0,\  gx-\omega^{-1}xg=0,\ gy-\omega yg=0,\ xy-\omega yx=0 \right\rangle.
\] 
$H$ is a Hopf algebra when it is equipped with coalgebra structure given by $g$ being a group-like element and $x,y$ $(g,1)$-primitive elements. 
The set $\{g^rx^py^s:\ 0\leq r,p,s\leq n-1\}$ is a basis of $H$, so $\dim H=n^3$.
We have $\ant^2(g)=g$, $\ant^2(x)=\omega^{-1} x$ and $\ant^2(y)=\omega y$, hence the order of $\ant^2$ is $n$. 

Using again the same method than before but with more labour, we can prove that $H$ is an almost involutive Hopf algebra.

Here we can also obtain four companion automorphisms $\sigma$ by direct inspection, defining $\sigma(g)=g$, $\sigma(x)=\pm \nu^{-1} x$ and $\sigma(y)=\pm \nu y$, being $\nu\in\k$ such that $\nu^2=\omega$.
\end{example}


\begin{example}[Hopf algebras of dimension 8] \label{ej:8}
If $H$ is a 8-dimensional not semisimple Hopf algebra, then Stefan shows in \cite{kn:stefan} that $H$ is isomorphic to one and only one of the following list
\[
A_{C_2},\quad A'_{C_4},\quad A''_{C_4},\quad A'''_{C_4,\omega},\quad A_{C_2\times C_2},\quad \left( A''_{C_4}\right)^*,
\]
being
\begin{enumerate}
\item \label{item:ac2}
$A_{C_2}=\left\langle g,x,y:\ g^2=1,\ x^2=0,\ y^2=0,\ gx + xg=0,\ gy+yg=0,\ xy+ yx=0 \right\rangle$, $g$ is a group-like element and $x,y$ are $(g,1)$-primitive elements.
\item \label{item:a1c4}
$A'_{C_4}=\left\langle g,x:\ g^4=1,\ x^2=0,\ gx + xg=0 \right\rangle$, $g$ is a group-like element and $x$ is a $(g,1)$-primitive element.
\item \label{item:a2c4}
$A''_{C_4}=\left\langle g,x:\ g^4=1,\ x^2=g^2-1,\ gx + xg=0 \right\rangle$, $g$ is a group-like element and $x$ is a $(g,1)$-primitive element.
\item \label{item:a3c4}
$A'''_{C_4,\omega}=\left\langle g,x:\ g^4=1,\ x^2=0,\ gx- \omega xg=0 \right\rangle$, $g$ is a group-like element, $x$ is a $(g,1)$-primitive element and $\omega\in\K$ is a primitive root of unity of order 4.
\item \label{item:a2c2c2}
$A_{C_2\times C_2}=\left\langle g,h,x:\ g^2=1,\ h^2=1,\ x^2=0,\ gx + xg=0,\ hx + xh=0,\ gh-hg=0 \right\rangle$, $g$ and $h$ are group-like elements and $x$ is a $(g,1)$-primitive element.
\end{enumerate}
The algebras $A_{C_2}, A'_{C_4}, A'''_{C_4,\omega}$ and $A_{C_2\times C_2}$ are almost involutive. We give the companion automorphism $\sigma$ by its values in the generators.
\begin{itemize}
\item
$A_{C_2}$, $\sigma(g)=g$, $\sigma(x)=ix$ and $\sigma(y)=iy$. Observe that this is the case $n=2$ in the example \ref{example:rad}.
\item
$A'_{C_4}$, $\sigma(g)=g$ and $\sigma(x)=ix$.
\item
$A'''_{C_4,\omega}$, $\sigma(g)=g$ and $\sigma(x)=\omega x$.
\item
$A_{C_2\times C_2}$, $\sigma(g)=g$, $\sigma(h)=h$ and $\sigma(x)=ix$.
\end{itemize}

Notice, that of these cases, the situation described in \eqref{item:a1c4}, \eqref{item:a3c4} and \eqref{item:a2c2c2} fit into the pattern of the results appearing in Subsection \ref{subsection:const}, \eqref{item:trivial}.

Now we consider the algebra $A''_{C_4}$. 
The set $\left\{1,g,g^2,g^3,x,gx,g^2x,g^3x \right\}$ is a basis of $A''_{C_4}$ with decomposition $E_{\ant^2,1}=\langle1,g,g^2,g^3\rangle_{\k}$, $E_{\ant^2,-1}=\langle x,gx,g^2x,g^3x\rangle_{\k}$ and with the following normalized integrals:
\[
\ell=\left(1+g+g^2+g^3\right)x 
,\quad r=\left(-1+g-g^2+g^3\right)x
,\quad \lambda=x^*
,\quad \rho=\left(g^3x\right)^*
.
\]
The modular element is $a=g$ and the modular function $\alpha$ is defined by $\alpha(g)=-1$ and $\alpha(x)=0$.

Suppose there exists a companion automorphism $\sigma$ in $A''_{C_4}$. From \eqref{eq:alfasigma} we get $\sigma(g)=g$. 
Then the condition $gx + xg=0$ implies $\sigma(x)=b_0x+b_1gx+b_2g^2x+b_3g^3x$, for some $b_0,b_1,b_2,b_3\in\K$.
Using \eqref{eq:erresigma} for $\lambda$ and $\rho$ we have $\sigma(x)=r_\sigma x+b_1gx+b_2g^2x$, where $r_\sigma^2=-1$. 
Now using \eqref{eq:erresigma} for $\ell$ and $r$ we conclude $\sigma(x)=r_\sigma x$.
So $\sigma$ verifies $\sigma(g)=g$ and $\sigma(x)=r_\sigma x$, but then $\sigma$ can not preserve the relation $x^2=g^2-1$. 
Hence we have shown that  $A''_{C_4}$ is not almost involutive. As the property of being almost involutive is preserved by duality, we deduce that $\left( A''_{C_4}\right)^*$ is also not almost involutive. 
Note that $A''_{C_4}$ is pointed but $\left( A''_{C_4}\right)^*$ it is not --see \cite{kn:stefan}--. 
\end{example}

\begin{example}[Hopf algebras of dimension 12] \label{ej:12}
Let $H$ be a Hopf algebra of dimension $12$. 
Natale shows in \cite{kn:natale} that if $H$ is non semisimple, then $H$ or $H^\vee$ is pointed; she also shows that if $H$ is pointed, then it is isomorphic to one and only one of the following list:
\begin{itemize}
\item
$A_{0}=\left\langle g,x:\ g^6=1,\ x^2=0,\ gx + xg=0 \right\rangle$, $g$ is a group-like element and $x$ is $(g,1)$-primitive.
\item
$A_{1}=\left\langle g,x:\ g^6=1,\ x^2=1-g^2,\ gx + xg=0 \right\rangle$, $g$ is a group-like element and $x$ is  $(g,1)$-primitive.
\item
$B_{0}=\left\langle g,x:\ g^6=1,\ x^2=0,\ gx + xg=0 \right\rangle$, $g$ is a group-like element and $x$ is $\left(g^3,1\right)$-primitive.
\item
$B_{1}=\left\langle g,x:\ g^6=1,\ x^2=0,\ gx -\omega xg=0 \right\rangle$, $g$ is a group-like element, $x$ is $\left(g^3,1\right)$-primitive and $\omega\in\K$ is a primitive root of unity of order 6.
\end{itemize}

The Hopf algebras in this list satisfy the following: $A_0^*=B_1$, $B_0^*=B_0$ and $A_1^*$ is not pointed. Moreover, $A_0$, $B_0$ and $B_1$ are of the type appearing in Subsection \ref{subsection:const}, \eqref{item:trivial}. 

The algebras in this list appear analogous to the ones in dimension 8, so we can expect that they have similar properties. Indeed, we have that $A_0$, $B_0$ and $B_1$ are almost involutive but $A_1$ --and so its dual-- is not.
The proof that $A_0$, $B_0$ and $B_1$ are almost involutive follows a similar pattern than the eight dimensional case.
For $A_1$, if it has a companion morphism $\sigma$, then similarly than for $A''_{C_4}$ we obtain that 
$\sigma(g)=g$ and also prove the existence of scalars $a,b\in\K$ such that $\sigma(x)=r_\sigma x + a\left(gx-g^3x \right)+b\left(gx^2-g^4x \right)$.
Being $\sigma$ a Hopf algebra map, then $\ant(\sigma(x))=\sigma(\ant(x))$, and this relation implies $\sigma(x)=r_\sigma x$ and we obtain the same contradiction than for $A''_{C_4}$.
\end{example}

\begin{rema}
The Hopf algebras of dimension 13, 14 and 15 are semisimple --see \cite{kn:beattie-gaston}-- and the ones of dimension $n\leq 11$ and $n\ne 8$ are semisimple or Taft algebras --see \cite{kn:stefan}--.
Semisimple Hopf algebras are involutive and in the example \ref{example:taft} below we show that the Taft algebra is almost involutive. 
Hence, of the Hopf algebras of dimension $n\leq 15$, the only cases when there are  non almost involutive examples is for $n=8$ or $n=12$.
\end{rema}

\section{Appendix: square roots of finite order linear automorphisms.}
We start with some elementary considerations about the square root of a finite order linear automorphism $D:V \rightarrow V$ where $V$ is a finite dimensional vector space over a field $\mathbb K$. 
We call $m=|D|$ the order of $D$. 

Given such $m$, we take  $r \in \mathbb K$ with the property that its order is $|r|=2m$, if $m$ is even or  $|r|=m$, if $m$ is odd.
We call $q=r^2$. Notice that $|q|=m$. 

Define $\mathcal E=\big\{0 \leq i \leq m-1: q^i \in \operatorname{Spec}(D)\big\}$; then $V=\bigoplus_{i \in \mathcal E}E_{D,q^i}$ where $E_{D,q^i}=\{x \in H: Dx=q^i x\}$.

Assume that $\sigma:V \rightarrow V$ is a linear automorphism of $V$ such that $\sigma^2=D$. 
Any such $\sigma$ will satisfy that $\sigma^2|_{E_{D,q^i}}=q^i\operatorname{id}$ and hence, the minimal polynomial $m_{\sigma|_{E_{D,q^i}}} | (t^2-q^i)=t^2-r^{2i}=(t-r^i)(t+r^i)$. 
Then for all $i \in \mathcal E$ we can find two subspaces $E_{\sigma,r^i},E_{\sigma,-r^i} \subseteq E_{D,q^i}$ --one of them could be $\{0\}$-- such that 
$E_{D,q^i}= E_{\sigma,r^i} \oplus E_{\sigma,-r^i}$ and $\sigma|_{E_{\sigma,\pm r^i}}=\pm r^i\operatorname{id}$.

Conversely, if for every $i \in \mathcal E$ we are given an arbitrary direct sum decomposition of $E_{D,q^i}=V_{+,i} \oplus V_{-,i}$ as above, 
then we can define an operator $\sigma: V\rightarrow V$, by requiring that its restriction to each of the summands are $\pm r^i\operatorname{id}$. 
In other  words, if we write $x \in E_{D,q^i}$ as $x=x_+ + x_-$ with $x_\pm \in V_{\pm,i}$, then $\sigma(x)=r^i x_{+} - r^i x_-$. 

By construction $\sigma^2=D$ on $E_{D,q^i}$ for all $i\in \mathcal E$ and then $\sigma$ is a square root of $D$ on all of $V$. For such $\sigma$ we have that for all $i \in \mathcal E$: $E_{\sigma,\pm r^i}=V_{\pm,i}$.
  
Hence, to define a \emph{linear transformation} that is a square root of $D$, we have to take for each eigenspace $E_{D,q^i}$ of $D$ with $i \in \mathcal E$, a decomposition on two subspaces $E_{D,q^i}=V_{+,i} \oplus V_{-,i}$. 
Given the decomposition, a square root is defined by the equations: $\sigma|_{V_{\pm,i}}=\pm r^i\operatorname{id}$.

\subsection{The case of an automorphism of Hopf algebras}

Assume now that $H$ is a finite dimensional Hopf algebra and that $D:H \rightarrow H$ is an \emph{automorphism} of Hopf algebras of order $|D|=m$. 

We want to find conditions for the pair of subspaces $V_{+,i}$ and $V_{-,i}$ that guarantee that the $\sigma$ thus defined is a Hopf algebra automorphism. The elementary results that we present below, are probably well known, we wrote them for the lack of an adequate reference. 

\begin{obse} \label{obse:Dcase}In the situation above with $H$ a finite dimensional Hopf algebra and $D$ a linear automorphism of finite order $m$. 
\begin{enumerate} 
\item $D$ is a morphism of algebras if and only if the following holds: 
\begin{enumerate}
\item $1 \in E_{D,1}$;\\ 
\item For $i,j \in \mathcal E$, $E_{D,q^i}E_{D,q^j} \subseteq \begin{cases}E_{D,q^{i+j}}\quad &\text{if}\quad i+j<m;\\ E_{D,q^{i+j-m}} &\text{if}\quad i+j \geq m.\end{cases}$
\end{enumerate}
\item $D$ is a morphism of coalgebras if and only if the following holds: 
\begin{enumerate}
\item If $0 \neq i \in \mathcal E$, then $\varepsilon(E_{D,q^i})=0$ ;
\item For $i \in \mathcal E$: 
\[
\Delta(E_{D,q^i}) \subseteq \bigoplus_{\{a,b \in \mathcal E: a+b=i\}}(E_{D,q^a} \otimes E_{D,q^b}) \quad \oplus \bigoplus_{\{a,b \in \mathcal E: a+b=i+m\}}(E_{D,q^a} \otimes E_{D,q^b}).
\]
\end{enumerate}
\end{enumerate}
Observe that in the considerations above we used that if $r,s \in \mathcal E$, then $0 \leq r,s \leq m-1$, and then $0 \leq r+s \leq 2m-2$. 
\end{obse}


\begin{theo}\label{theo:algebraconditions} 
Let $A$ be an algebra and let $D:A \rightarrow A$ be an automorphism of algebras of finite order $m$. Define as above, $q,r$, $\mathcal E$ and $E_{D,q^i}$ --for $i \in \mathcal E$--. For each $i \in \mathcal E$ when we take an arbitrary decomposition of $E_{D,q^i}=V_{+,i}\oplus V_{-,i}$ and define a linear transformation $\sigma$ on $A$ as: $\sigma|_{V_{\pm,i}}=\pm r^i\operatorname{id}$ for all $i \in \mathcal E$, then $\sigma^2=D$. Moreover, $\sigma$ is an automorphism of algebras if and only if the following conditions hold:
\begin{enumerate}
\item $1 \in V_{+,0}$\emph{;}
\item 
       \begin{enumerate}
\item If $0 \leq i+j \leq m-1$, then $V_{+,i} V_{+,j}+ V_{-,i}V_{-,j} \subseteq V_{+,i+j}$ and $V_{+,i} V_{-,j}+ V_{-,i}V_{+,j} \subseteq V_{-,i+j}$.
\item If $m \leq i+j \leq 2m-2$ then\emph{:}
               \begin{enumerate}
\item If $m$ is even, then $V_{+,i} V_{+,j}+ V_{-,i}V_{-,j} \subseteq V_{-,i+j-m}$ and $V_{+,i} V_{-,j}+ V_{-,i}V_{+,j} \subseteq V_{+,i+j-m}$\emph{;}
\item If $m$ is odd, then $V_{+,i} V_{+,j}+ V_{-,i}V_{-,j} \subseteq V_{+,i+j-m}$ and $V_{+,i} V_{-,j}+ V_{-,i}V_{+,j} \subseteq V_{-,i+j-m}$.   
               \end{enumerate}
         \end{enumerate}
\end{enumerate}
\end{theo}
\begin{proof}
\begin{enumerate}
\item \emph{Conditions for the unit.} The unit element, $1 \in E_{D,q^0}=E_{D,1}$ and as we want that $\sigma(1)=1$, in the decomposition of $E_{D,1}=V_{+,0} \oplus V_{-,0}$, $1 \in V_{+,0}$. 
\item \emph{Multiplicativity.} It is enough to prove that for all $i,j \in \mathcal E, x \in V_{\pm,i}, y \in V_{\pm,j} \Rightarrow \sigma(xy)=\sigma(x)\sigma(y)$. Being $D(xy)=q^{i+j}xy$, we have two alternatives:
\begin{enumerate}
\item If $xy=0$, then $\sigma(x)\sigma(y)=(\pm r^ix)(\pm r^jy)=0=\sigma(xy)$.
\item If $xy \neq 0$, $q^{i+j}\in \operatorname{Spec}(D)$ and $\exists k \in \mathcal E, i+j \equiv k(\!\!\!\mod m)$. 

Then: $k= \begin{cases}i+j\quad &\text{for}\,\, 0 \leq i+j \leq m-1;\\
i+j-m\quad &\text{for}\,\, m \leq i+j\leq 2m-2.\end{cases}$

As $xy \in E_{D,q^k}=V_{+,k} \oplus V_{-,k}$ for $k \in \mathcal E$, we can find $(xy)_\pm \in V_{\pm,k}$, such that: 
$xy=(xy)_++(xy)_-\,,\, \sigma(xy)=r^k(xy)_+-r^k (xy)_{-}$.\\
We consider the following alternatives.\\
(A)  $x \in V_{+,i}$ and $y \in V_{+,j}$ or $x \in V_{-,i}$ and $\in V_{-,j}$. 

In this case $\sigma(x)\sigma(y)=r^{i+j}xy=r^{i+j}(xy)_{+}+r^{i+j}(xy)_{-}$.
Now, if $i+j=k \leq m-1$ then $r^{i+j}=r^k$ and in accordance with the above formul\ae\  the multiplicativity holds if and only if $(xy)_{-}= 0$.

If $m \leq i+j=k+m$, then $r^{i+j}=r^kr^m=\begin{cases}-r^k &\quad \text{if $m$ is even}\\ \,\,\,\, r^k &\quad \text{if $m$ is odd}.\end{cases}$\\
Then, the multiplicativity holds if and only if
$\begin{cases}(xy)_+=0 &\,\text{if $m$ is even}\\ (xy)_-=0 &\,\text{if $m$ is odd}.\end{cases}$\\
(B) $x \in V_{+,i}$ and $y \in V_{-,j}$ or $x \in V_{-,i}$ and $y \in V_{+,j}$.
In this case $\sigma(x)\sigma(y)=-r^{i+j}xy=-r^{i+j}(xy)_+-r^{i+j}(xy)_-$. Now, if $i+j=k \leq m-1$ then $r^{i+j}=r^k$ and the multiplicativity holds if and only if $(xy)_+ = 0$.

If $m \leq i+j=k+m$, then $r^{i+j}=r^kr^m=\begin{cases}-r^k &\quad \text{if $m$ is even}\\ \,\,\,\,r^k &\quad \text{if $m$ is odd}.\end{cases}$\\
Hence, in this situation the multiplicativity holds if and only if
$\begin{cases}(xy)_-=0 &\,\text{if $m$ is even}\\ (xy)_+=0 &\,\text{if $m$ is odd}.\end{cases}$
\end{enumerate}
\end{enumerate}
\end{proof}


As we are dealing with finite dimensional objects, we may proceed by duality and obtain the following result:
  
\begin{theo}\label{theo:coalgebraconditions} 
Let $C$ be a coalgebra and let $D:C \rightarrow C$ be a automorphism of coalgebras of finite order $m$. 
Define as above, $q,r$, $\mathcal E$ and $E_{D,q^i}$ --for $i \in \mathcal E$--. 
For each $i \in \mathcal E$ when we take an arbitrary decomposition of $E_{D,q^i}=V_{+,i}\oplus V_{-,i}$ and define a linear transformation $\sigma$ on $C$ as: $\sigma|_{V_{\pm,i}}=\pm r^i\operatorname{id}$, then $\sigma^2=D$. 

  Moreover, $\sigma$ is an automorphism of coalgebras if and only if the following conditions hold:
  \begin{enumerate}
  \item $\varepsilon(V_{-,0})=0$ 
  \item 
\[\Delta(V_{+,i}) \subseteq \bigoplus_{\{a,b\in \mathcal E:a+b=i\}}(V_{+,a} \otimes V_{+,b} + V_{-,a} \otimes V_{-,b}) \oplus \begin{cases}\bigoplus_{\{a,b\in \mathcal E:a+b=i+m\}}(V_{+,a} \otimes V_{-,b} + V_{-,a} \otimes V_{+,b})\,,\, \text{m even};\\\bigoplus_{\{a,b\in \mathcal E:a+b=i+m\}}(V_{+,a} \otimes V_{+,b} + V_{-,a} \otimes V_{-,b})\,,\, \text{m odd}.\end{cases}\]
\[\Delta(V_{-,i}) \subseteq \bigoplus_{\{a,b\in \mathcal E:a+b=i\}}(V_{+,a} \otimes V_{-,b} + V_{-,a} \otimes V_{+,b}) \oplus \begin{cases}\bigoplus_{\{a,b\in \mathcal E:a+b=i+m\}}(V_{+,a} \otimes V_{+,b} + V_{-,a} \otimes V_{-,b})\,,\, \text{m even};\\\bigoplus_{\{a,b\in \mathcal E:a+b=i+m\}}(V_{+,a} \otimes V_{-,b} + V_{-,a} \otimes V_{+,b})\,,\, \text{m odd}.\end{cases}
\]
\qed
  \end{enumerate}
 \end{theo}

\begin{coro} \label{coro:Hopf-conditions}
Let $H$ be a Hopf algebra and let $D:H \rightarrow H$ be an automorphism of Hopf algebras of finite order $m$. 
Define as above, $q,r$, $\mathcal E$ and $E_{D,q^i}$ --for $i \in \mathcal E$--. 
For each $i \in \mathcal E$ we take an arbitrary decomposition of $E_{D,q^i}=V_{+,i}\oplus V_{-,i}$ and define a linear transformation $\sigma$ on $C$ as: $\sigma|_{V_{\pm,i}}=\pm r^i\operatorname{id}$. 
If the hypothesis of theorems \ref{theo:algebraconditions} and \ref{theo:coalgebraconditions} are simultaneously satisfied, then $\sigma$ is a Hopf algebra automorphism and $\sigma^2=D$. \qed
\end{coro}

\subsection{A particular situation.}
We consider the following special cases of the above Corollary \ref{coro:Hopf-conditions}.
Assume that the splitting of the eigenspaces $E_{D,q^i}$ is trivial: \[V_{+,i}=E_{D,q^i}\,\,\,\,\text{and}\quad V_{-,i}=0.\] 

In this situation, and using the considerations of Observation \ref{obse:Dcase}, it is clear that some of the conditions of Corollary \ref{coro:Hopf-conditions} --i.e. 
of Theorems \ref{theo:algebraconditions} and \ref{theo:coalgebraconditions}-- are authomatically verified. In particular the case $m$ odd becomes conditionless.
Hence, we have the following particular result that in some cases provides an answer for the existence of a square root of a Hopf automorphism that is both multiplicative and comultiplicative.

\begin{coro}\label{coro:sqrpart} 
Assume that we are in the situation above. 
If $m$ is odd, then the square root of $D$ associated to the family of subspaces $V_{+,i}=E_{D,q^i}$ and $V_{-,i}=0$  is an automorphism of Hopf algebras. 
If $m$ is even, it is an automorphism of Hopf algebras if and only if $E_{D,q^i} E_{D,q^j}=0 $ for all $i,j \in \mathcal E$ such that $m \leq i+j \leq 2m-2$ and 
$\Delta(E_{D,q^i}) \subseteq \bigoplus_{\{a,b\in \mathcal E:a+b=i\}}(E_{D,q^a} \otimes E_{D,q^b})$ for all $i\in \mathcal E$. 
\qed
\end{coro}

Observe that the case in which $m$ is odd has already been treated by an elementary reasonement in Observation \ref{obse:oddtrivial}.

\begin{example}[Taft algebra]\label{example:taft}
The Taft algebra $T_n$ is a generalization of the Sweedler's algebra $H_4$.
Let be $\omega\in\k$ a primitive root of order $n$ of 1 and 
\[
T_n=\left\langle g,x:\ g^n=1,\ x^n=0,\ gx-\omega xg=0\right\rangle.
\] 
$T_n$ is a Hopf algebra with coalgebra structure given by $g$ being a group-like element and $x$ a $(g,1)$-primitive element.

\noindent 
The set $\{g^rx^p:\ 0\leq r,p\leq n-1\}$ is a basis of $T_n$, so $\dim T_n=n^2$. 
We have $\ant^2(g)=g$ and $\ant^2(x)=\omega x$, therefore the order of $\ant^2$ is $n$.
\noindent
The eigenvalues of $\ant^2$ are $\{1,\omega,\cdots,\omega^{n-1}\}$, and the corresponding eigenspaces are $E_{\ant^2,\omega^i}=\{x^i,gx^i,\cdots,g^{n-1}x^i\}_{\mathbb K}$. 

\noindent
With regard of the conditions of Corollary \ref{coro:sqrpart}, we have that $E_{\ant^2,\omega^i}E_{\ant^2,\omega^j}=\{g^k \omega^{i+j}:k=0,\cdots,n-1\}$.  

\noindent
Hence $E_{\ant^2,\omega^i}E_{\ant^2,\omega^j}=0$ for $i+j \geq n$ as required. 

\noindent
Morever, as $\Delta(g^kx^i)=(g^k \otimes g^k)(x \otimes g + 1 \otimes x)^i$, it is clear that $\Delta(E_{\ant^2,\omega^i}) \subseteq \sum_{a+b=i} E_{\ant^2,\omega^a} \otimes E_{\ant^2,\omega^b}$.

\noindent
Hence, in this manner we prove that $T_n$ is almost involutive. 

\noindent
A direct verification shows that if we take $\nu\in\k$, $\nu^2=\omega$. The maps $\sigma_\pm$ defined as $\sigma_\pm(g)=g$ and $\sigma_\pm(x)=\pm \nu x$ are companion automorphisms.
\end{example}


\begin{example}\label{example:general}

\noindent
Assume that we have a trivial extension of a finite dimensional Hopf algebra, with the additional property that $K=E_1$. The spectral decomposition of $H$ with respect to $\ant^2$ becomes: 
\[
H= E_{\ant^2,1} \oplus \bigoplus_{i \in \mathcal E_M}E_{\ant^{2},q^i} \quad \text{with}\quad  M=\bigoplus_{i \in \mathcal E_M}E_{\ant^{2},q^i}.
\]
\noindent
In this case it is clear that the conditions of Corollary \ref{coro:sqrpart} are satisfied.

\noindent
Indeed, if we look at the condition regarding the product, the only cases in which the sum of exponents of the corresponding eigenvalues 
of two eigenvectors may be larger than $m$, is for the case that the exponents of the eigenvectors are in $\mathcal E_M$.  In this case, the condition $M^2=0$, guarantees that the product of the corresponding eigenspaces is trivial. 

\noindent
An argument along the same lines and using the condition that $M$ is a $K$--bicomodule, shows that the condition regarding the coproduct in Corollary \ref{coro:sqrpart} is satisfied.

\end{example}

\end{document}